\documentclass[11pt,leqno]{article} 
\usepackage[applemac]{inputenc}
\usepackage[pdftex]{graphicx}

\usepackage{hyperref} 
\usepackage{xcolor}
	\hypersetup{
    colorlinks,
    linkcolor={red!50!black},
    citecolor={blue!50!black},
    urlcolor={blue!80!black}
}
\usepackage{xy}
\usepackage{xypic} 
\usepackage{multicol}
\usepackage{xcolor}
\usepackage{amsmath}
\usepackage{amssymb}
\usepackage{amsthm} 
\usepackage{mathrsfs}

\def\C{{\mathbb C}} 
\def\N{{\mathbb N}}
\def\Q{{\mathbb Q}}
\def\R{{\mathbb R}}
\def\Z{{\mathbb Z}}
\def\scrS{{\mathscr S}}
\def\calE{{\mathcal E}}
 
\def\bfs{{\mathbf{s}}}
\def\bfw{{\mathbf{w}}}
\def\bfz{{\mathbf{z}}}

\def\Ltilde{\widetilde{L}}
\def\Mtilde{\widetilde{M}}
\def\Lambdatilde{\widetilde{\Lambda}} 

\def\rmD{{\mathrm D}}
\def\rmd{{\mathrm d}}
\def\rme{{\mathrm e}}
 
\numberwithin{equation}{section}

\newtheorem{theorem}{Theorem}
\newtheorem{corollary}[theorem]{Corollary}
\newtheorem{proposition}[theorem]{Proposition}
\newtheorem{lemma}[theorem]{Lemma}
\newtheorem*{remark}{Remark}

\begin{document} 
 
 \null
 \vskip -3 true cm
 
 \hfill \hfill \hfill \hfill \hfill
 {\footnotesize 
 To appear in the Moscow Journal of Combinatorics and Number Theory

} 
 
 \vskip 3 true cm
 
 \begin{center}
 \LARGE
 \bf
 
 On transcendental entire functions 
 
 with infinitely many derivatives 
 taking 
 
 integer values at 
several points
 
 \large
 \medskip
 by
 
 \medskip
 \sc
 Michel Waldschmidt
 \end{center}

 
 Let $s_0,s_1,\dots,s_{m-1}$ be 
complex numbers and $r_0,\dots,r_{m-1}$ rational integers in the range $0\le r_j\le m-1$. Our first goal is to prove that if an entire function $f$ of sufficiently small exponential type satisfies $f^{(mn+r_j)}(s_j)\in\Z$ for $0\le j\le m-1$ and all sufficiently large $n$, then $f$ is a polynomial. Under suitable assumptions on $s_0,s_1,\dots,s_{m-1}$ and $r_0,\dots,r_{m-1}$, we introduce interpolation polynomials $\Lambda_{nj}$, ($n\ge 0$, $0\le j\le m-1$) satisfying 
$$
\Lambda_{nj}^{(mk+r_\ell)}(s_\ell)=\delta_{j\ell}\delta_{nk}, \quad\hbox{for}\quad n, k\ge 0 \quad\hbox{and}\quad 0\le j, \ell\le m-1
$$
and we show that any entire function $f$ of sufficiently small exponential type has a convergent expansion
$$
f(z)=\sum_{n\ge 0} \sum_{j=0}^{m-1}f^{(mn+r_j)}(s_j)\Lambda_{nj}(z).
$$ 
The case $r_j=j$ for $0\le j\le m-1$ involves successive derivatives $f^{(n)}(w_n)$ of $f$ evaluated at points of a periodic sequence  $\bfw=(w_n)_{n\ge 0}$ of complex numbers, where $w_{mh+j}=s_j$ ($h\ge 0$, $0\le j\le m$). 
More generally, given a bounded  (not necessarily periodic) sequence $\bfw=(w_n)_{n\ge 0}$ of complex numbers, we consider similar interpolation formulae 
$$
f(z)=\sum_{n\ge 0}f^{(n)}(w_n)\Omega_{\bfw,n}(z)
$$ 
involving polynomials $\Omega_{\bfw,n}(z)$ which were introduced by W.~Gontcharoff in 1930. Under suitable assumptions, we show that the hypothesis $f^{(n)}(w_n)\in\Z$ for all sufficiently large $n$ implies that $f$ is a polynomial.


\section{Introduction}\label{S:Introduction}

Given a finite set of points $S$ in the complex plane and an infinite subset $\scrS$ of $S\times \N$, where $\N=\{0,1,2,\dots\}$ is the set of nonnegative integers, we ask for a lower bound on the order of growth of a transcendental entire function $f$ such that $f^{(n)}(s)\in\Z$ for all $(s,n)\in\scrS$. 
In \cite{MW-Lidstone2points}, we discussed the case $S=\{s_0,s_1\}$ using interpolation polynomials of Lidstone, Whittaker and Gontcharoff, together with results of Schoenberg and Macintyre. 

 Here we introduce generalizations of these interpolation polynomials to several points and we deduce lower bounds for the growth of transcendental entire functions with corresponding integral values of their derivatives. We first consider periodic sequences: given complex numbers $s_0,s_1,\dots,s_{m-1}$ and rational integers $r_0,\dots,r_{m-1}$ in the range $0\le r_j\le m-1$, we set 
$$
\scrS=\{(s_j,mn+r_j)\; \mid \; n\ge 0,\; 0\le j\le m-1\};
$$ 
under suitable assumptions, we give a lower bound for the growth order of a transcendental entire function $f$ satisfying $f^{(mn+r_j)}(s_j)\in\Z$ for $0\le j\le m-1$ and all sufficiently large $n$ (Theorem \ref{Theorem:mpointsperiodic}). That some assumption is necessary is obvious from the example $m=2$, $s_0=s_1=r_0=r_1=0$: given any transcendental entire function $g$, say of order $0$, the function $f(z)=zg(z^2)$ is a transcendental entire function of the same order satisfying $f^{(2n)}(s_0)=0$ for all $n\ge 0$. 

Next, we consider a sequence $(w_n)_{n\ge 0}$ of elements in $S$ and we prove that an entire function of sufficiently small exponential type satisfying $f^{(n)}(w_n)\in\Z$ for all sufficiently large $n$ is a polynomial (Theorem \ref{Th:mPointsAlln}(a)). 

In Section \ref{Section:SeveralPointsPeriodic}, we show how to interpolate entire functions of sufficiently small exponential type with respect to periodic subsets of $\{s_0, s_1,\dots,s_{m-1}\} \times \N$. Our approach requires that some determinant $D(s_0, s_1,\dots,s_{m-1})$ (depending also on $r_0,\dots,r_{m-1}$) does not vanish; this assumption cannot be omitted (it could be weakened, but we do not address this issue here).  
 
In Section \ref{S:ComplementarySequences}, we introduce interpolation polynomials attached to a sequence of elements belonging to $\{s_0, s_1,\dots,s_{m-1}\}$. We deduce 
 that if $f$ is an entire function $f$ of sufficiently small exponential type such that, for all sufficiently large $n$, one at least of the $2^m-1$ nonempty products of elements $f^{(n)}(s_0), f^{(n)}(s_1),\dots,f^{(n)}(s_{m-1})$ is in $\Z$, then $f$ is a polynomial (Theorem \ref{Th:mPointsAlln}(b)). 
 
\section{Notation and auxiliary results}\label{S:Notations}
We denote by $\delta_{ij}$ the Kronecker symbol: 
$$
\delta_{ij}=\begin{cases}
1 & \text{ if $i=j$},
\\
0 & \text{ if $i\neq j$,}
\end{cases}
$$
and by $f^{(n)}$ the $n$--th derivative $(\rmd^n /\rmd z^n) f$ of an analytic function $f(z)$. 
 
The order of an entire function $f$ is 
$$
\varrho(f)=\limsup_{r\to\infty} \frac{\log\log |f|_r}{\log r}
\;
\text{ where } 
\;
|f|_r=\sup_{|z|=r}|f(z)|,
$$
and the exponential type is 
$$
\tau(f)= \limsup_{r\to\infty} \frac{ \log |f|_r}{r}\cdotp
$$
For each $z_0\in\C$, we have
\begin{equation}\label{eq:type}
\limsup_{n\to\infty} |f^{(n)}(z_0)|^{1/n}= \tau(f).
\end{equation}

Cauchy's inequalities 
\begin{equation}\label{Equation:CauchyInequality}
\frac{|f^{(n)}(z_0)|}{n!}r^{n} \le |f|_{r+|z_0|},
\end{equation}
 are valid for any entire function $f$ and all $z_0\in\C$, $n\ge 0$ and $r>0$. We will also use Stirling's Formula: for all $N\ge 1$, we have
\begin{equation}\label{Equation:Stirling}
N^N\rme^{-N}\sqrt{2\pi N}
< N! <
N^N \rme^{-N} \sqrt{2\pi N} \rme^{1/(12N)}.	
\end{equation}

For the arithmetical applications, our main assumption on the growth of our functions $f$ is
\begin{equation}\label{eq:maingrowthconditionmpoints}
\limsup_{r\to\infty}\rme^{-r}\sqrt r |f|_r<\frac{1}{\sqrt{2\pi}}
\rme^{-\max\{|s_0|,| s_1|,\dots,|s_{m-1}|\}}.
\end{equation}
This condition arises from the following auxiliary result, based on Cauchy's upper bound for the derivatives and Stirling approximation formula for $n!$ 
 \cite[Proposition 12]{MW-Lidstone2points}: 
 
\begin{proposition}\label{Proposition:Polya}
Let $f$ be an entire function and let $A\ge 0$. Assume 
\begin{equation}\label{eq:maingrowthconditionA}
\limsup_{r\to\infty}\rme^{-r}\sqrt r |f|_r< \frac{ \rme^{-A}}{\sqrt{2\pi} }\cdotp
\end{equation}
Then there exists $n_0>0$ such that, for $n\ge n_0$ and for all $z\in\C$ in the disc $|z|\le A$, we have 
$$
 |f^{(n)}(z)|<1.
$$ 
\end{proposition}

\section{Integer values of derivatives of entire functions}\label{S:main}

\subsection{Periodic sequences}  \label{SS:periodicm}

Let $s_0,s_1,\dots,s_{m-1}$ be complex numbers, not necessarily distinct. We write $\bfs$ for the tuple $(s_0,s_1,\dots,s_{m-1})$. Let $\zeta$ be a primitive $m$-th root of unity and let $r_0,\dots,r_{m-1}$ be $m$ integers satisfying $0\le r_j\le j$ ($0\le j\le m-1$). 
Our main assumption is that the determinant 
$$
\rmD(\bfs)=
\det
\left(
\frac{k!}{(k-r_j)!} s_j^{k-r_j}
\right)_{0\le j,k\le m-1}
$$
does not vanish.
Here, $ a!/(a-b)!$ is understood to be $0$ for $a<b$. This assumption means that the linear map 
\begin{equation}\label{Equation:D(s)not0}
\begin{matrix}
\C[z]_{\le m-1}&\longrightarrow& \C^m \qquad
\\
L(z)&\longmapsto& \bigl(L^{(r_j)}(s_j)\bigr)_{0\le j\le m-1}
\end{matrix}
\end{equation}
is an isomorphism of $\C$--vector spaces, $\C[z]_{\le m-1}$ being the space of polynomials of degree $\le m-1$. 

 For $t\in\C$, consider the $m\times m$ matrix
$$
M(t)=\left(\zeta^{kr_\ell} \rme^{\zeta^k t s_\ell}\right)_{0\le k, \ell\le m-1}
$$ 
and its determinant $\Delta(t)$:
$$
\Delta(t) =
\det
\begin{pmatrix}
\rme^{ts_0}&\rme^{ts_1}&\cdots&\rme^{t s_{m-1}}
\\
\zeta^{r_0} \rme^{\zeta ts_0}&\zeta^{r_1}\rme^{\zeta ts_1}&\cdots& \zeta^{r_{m-1}} \rme^{\zeta t s_{m-1}}
\\
\vdots&\vdots&\ddots&\vdots
\\
\zeta^{(m-1)r_0} \rme^{\zeta^{m-1} ts_0}&\zeta^{(m-1)r_1}\rme^{\zeta^{m-1} ts_1}&\cdots&\zeta^{(m-1)r_{m-1}}\rme^{\zeta^{m-1} t s_{m-1}}
\\
\end{pmatrix}.
$$
We will show (Lemma \ref{Lemma:Deltanot0}) that the exponential polynomial $\Delta(t)$ is not the zero function. 
%


\begin{theorem}\label{Theorem:mpointsperiodic}
Assume $\rmD(\bfs)\not=0$. 
Let $\tau>0$ be such that $\Delta(t)$ does not vanish for $0<|t|<\tau$.  
Let $f$ be an entire function of exponential type $<\tau$ which satisfies \eqref{eq:maingrowthconditionmpoints}
and also, for each $n$ sufficiently large, 
$$
f^{(mn+r_j)}(s_j)\in\Z \text{ for } j=0,\dots,m-1.
$$
Then $f$ is a polynomial. 
\end{theorem}

In the case $m=1$, we can take $\tau=1$ and the assumption that the exponential type is $<1$ can be replaced by the weaker condition \eqref{eq:maingrowthconditionA} with $A=0$, according to a classical result of P\'olya on Hurwitz functions; see \cite[\S 2]{MW-Lidstone2points}. 

Let us give two further examples. Proofs will be given in Section \ref{SS:TheDeterminant}.

Our first example is with $r_0=r_1=\dots=r_{m-1}=0$. In this case, the assumption $\rmD(\bfs)\not=0$ is satisfied if and only if $s_0,s_1,\dots,s_{m-1}$ are pairwise distinct (Section \ref{SS:TheDeterminant} Example 1). 
 
\begin{corollary}\label{corollary:mcongru0}
Assume that $s_0,s_1,\dots,s_{m-1}$ are pairwise distinct.
An entire function of sufficiently small exponential type, satisfying 
$$
f^{(mn)}(s_j)\in\Z 
$$
for $j=0,\dots,m-1$ and for all sufficiently large $n$, is a polynomial. 
\end{corollary} 

 For $m=2$ (Lidstone interpolation), 
 with $f^{(2n)}(s_0)\in\Z$ and $f^{(2n)}(s_1)\in\Z$, 
 Corollary \ref{corollary:mcongru0} follows also from \cite[Corollary 2]{MW-Lidstone2points}, 
where 
 the assumption on the exponential type $\tau(f)$ of $f$ is 
 $$
 \tau(f)<\min\{1, \pi/|s_0-s_1|\},
 $$
and this assumption is best possible. Indeed, 
 
 \begin{itemize}
 \item
 The function 
 $$
 f(z)=\frac{\sinh(z- s_1)}{\sinh(s_0- s_1)}
 $$ 
has exponential type $1$ and satisfies 
$f^{(2n)}(s_0)=1$ and $f^{(2n)}( s_1)=0$ for all $n\ge 0$. 
 \item
 The function 
$$
f(z)=\sin\left( \pi\frac{z-s_0}{ s_1-s_0}\right)
$$ 
has exponential type $\pi/|s_1-s_0|$ and satisfies $f^{(2n)}(s_0)=f^{(2n)}( s_1)=0$ for all $n\ge 0$.
\end{itemize}

Our second example is $r_j=j$ for $j=0,1,\dots,m-1$. The assumption $\rmD(\bfs)\not=0$ is always satisfied (Section \ref{SS:TheDeterminant} Example 2). 

\begin{corollary}\label{corollary:mcongru01m-1}
An entire function of sufficiently small exponential type satisfying 
$$
f^{(mn+j)}(s_j)\in\Z 
$$
for $j=0,\dots,m-1$ and for all sufficiently large $n$ is a polynomial. 
\end{corollary}
 
 In the case $m=2$ (Whittaker interpolation), with 
 $f^{(2n+1)}(s_0)\in\Z$ and $f^{(2n)}(s_1)\in\Z$, Corollary \ref{corollary:mcongru01m-1} also follows from \cite[Corollary 6]{MW-Lidstone2points} 
(after permutation of $s_0$ and $s_1$), where
the assumption is 
 $$
 \tau(f)<\min\left\{1, \frac{\pi}{2|s_0-s_1|}\right\},
 $$ 
and this assumption is best possible. Indeed, 
 
 \begin{itemize}
 \item The function 
 $$
f(z)= \frac{\sinh(z- s_0)}{\cosh(s_1- s_0)}
$$  
has exponential type $1$ and satisfies 
$f^{(2n)}( s_0)=0$ and $f^{(2n+1)}(s_1)=1$ for all $n\ge 0$.
 
 \item
 The function 
 $$
f(z)=\cos\left(\frac{\pi}{2}\cdot \frac{z-s_1}{ s_1-s_0}
\right)
$$
has exponential type $ \pi/(2| s_1-s_0|)$ and satisfies $f^{(2n)}(s_0)=f^{(2n+1)}( s_1)=0$ for all $n\ge 0$.
\end{itemize}

\subsection{Sequence of derivatives at finitely many points}
 
The next result deals with a situation more general than Corollary \ref{corollary:mcongru01m-1}. 
 
\begin{theorem}\label{Th:mPointsAlln}
Let $A>0$, let $f$ be an entire function satisfying 
\eqref{eq:maingrowthconditionA} and let the exponential type $\tau(f)$ of $f$ satisfy
$$
\tau(f)< \frac {\log 2} A\cdotp
$$
(a)
Assume that for all sufficiently large integers $n$, there exists $w_n\in\C$ with $|w_n|<A$ such that $f^{(n)}(w_n)\in\Z$. Then $f$ is a polynomial. 
\\
(b)
Let $s_0, s_1,\dots,s_{m-1}$ be $m$ 
complex numbers, not necessarily distinct, satisfying 
$$
\max_{0\le j\le m-1}|s_j|<A.
$$
Assume that, for all sufficiently large $n$, there exists a nonempty subset $I_n$ of $\{0,1,\dots,m-1\}$ such that the product
$$
\prod_{j\in I_n} f^{(n)}(s_j) 
$$
 is in $\Z$. 
Then $f$ is a polynomial. 
\end{theorem}
 
The case $m=2$ in part (b) of Theorem \ref{Th:mPointsAlln} is \cite[Theorem 8]{MW-Lidstone2points}.
 
\subsection{Content}

 In Section \ref{Section:SeveralPointsPeriodic} we deal with periodic subsets of 
$S\times \N$: we generalize the construction of Lidstone polynomials to several points and we prove Theorem \ref{Theorem:mpointsperiodic} and Corollaries \ref{corollary:mcongru0} and \ref{corollary:mcongru01m-1}. 
In Section \ref{S:ComplementarySequences}, we introduce and study interpolation polynomials associated with a sequence of elements in $S$ and we prove Theorem \ref{Th:mPointsAlln}. 
 
\section{Periodic case}\label{Section:SeveralPointsPeriodic}

Let $s_0,s_1,\dots,s_{m-1}$ be distinct complex numbers and $r_0,\dots,r_{m-1}$ rational integers satisfying $0\le r_0\le r_1\le \cdots \le r_{m-1}\le m-1$.

\subsection{The determinant $\rmD(\bfz)$ -- proofs of Corollaries \ref{corollary:mcongru0} and \ref{corollary:mcongru01m-1}}\label{SS:TheDeterminant}

Let $z_0,z_1,\dots,z_{m-1}$ be independent variables. Write $\bfz$ for $(z_0,z_1,\dots,z_{m-1})$. Let $K$ be the field $
\Q(z_0,z_1,\dots,z_{m-1})$ and $\rmD(\bfz)$ be the determinant 
$$
\det
\left(
\frac{k!}{(k-r_j)!} z_j^{k-r_j}
\right)_{0\le j,k\le m-1} \in \Q[\bfz]\subset K.
$$
Recall $a!/(a-b)!=0$ for $a<b$.
 
For $j=0,1,\dots,m-1$, the row vector 
\begin{align}\notag
 v_j&=\left(
\frac{k!}{(k-r_j)!} z_j^{k-r_j}
\right)_{k=0,1,\dots, m-1} 
\\
\notag
&=
\left(0,0,\dots,0,r_j!,\frac{(r_j+1)!}{1!} z_j,\frac{(r_j+2)!}{2!} z_j^2,\dots,\frac{(m-1)!}{(m-1-r_j)!} z_j^{m-1-r_j} \right)
\end{align} 
belongs to $\{0\}^{r_j}\times K^{m-r_j}$. 
If $r_j>j$ for some $j\in\{0,1,\dots,m-1\}$, then the $m-j$ vectors $v_j,v_{j+1},\dots,v_{m-1}$ all belong to the subspace $\{0\}^{j+1}\times K^{m-j-1}$ of $K^m$, the dimension of which is $m-j-1$; hence the determinant $\rmD(\bfz)$ vanishes. 

Assume $r_j\le j$ for $0\le j\le m-1$. For the degree given by the lexicographic order, the leading term of the polynomial $\rmD(\bfz)$ is  the product of the elements on the diagonal. The degree in $z_j$ of $\rmD(\bfz)$ is $\le m-1-r_j$. 
For $k=0,1,\dots,m-1$, define 
$$
\calE(k)=\{(i,j)\; \mid \; 0\le i<j\le m-1,\; r_i=r_j\}.
$$ 
In the ring $\Q[z_0,z_1,\dots,z_{m-1}]$, the polynomial $\rmD(\bfz)$ is divisible by 
$$
\prod_{(i,j)\in \calE(k)} (z_j-z_i). 
$$
If there is no extra nonconstant factor, the only zeros of $\rmD(\bfz)$ are given by $z_i=z_j$ with $r_i=r_j$ and $i<j$. But extra factors may occur. 
 
 \medskip
\noindent
{\bf Examples}
\\
(1) \cite{MR1501639}, 
quoted by \cite[\S 3]{MR0059366} 
and \cite{MR0072947}: 
$$
r_0=r_1=\cdots=r_{m-1}=0.
$$ 
The Vandermonde determinant
$$
D(\bfs)=
\det
\left( s_j^k
\right)_{0\le j,k\le m-1}
=
\det
\left(
\begin{matrix}
1& s_0& s_0^2 & \cdots&s_0^{m-1}
\\
1& s_1& s_1^2 & \cdots&s_1^{m-1}
\\
1& s_2& s_2^2 & \cdots&s_2^{m-1}
\\
\vdots& \vdots& \vdots &\ddots&\vdots
\\
1& s_{m-1}& s_{m-1}^2 & \cdots&s_{m-1}^{m-1}
\\
\end{matrix}
\right)
=
\prod_{0\le j<\ell\le m-1}(s_\ell-s_j)
$$
does not vanish if and only if $s_0,s_1,\dots,s_{m-1}$ are pairwise distinct. 
\\
(2) \cite{zbMATH02563461}, 
quoted by \cite[\S 4]{MR0059366} 
and \cite{MR0072947}: 
$$
r_j=j \quad\hbox{for}\quad j=0,1,\dots,m-1.
$$ 
Then
\begin{align*}
D(\bfs)&=
\det
\left(
\frac{k!}{(k-j)!} s_j^{k-j}
\right)_{0\le j,k\le m-1}
\\
&=
\det
\left(
\begin{matrix}
1& s_0& s_0^2 &s_0^3& \cdots&s_0^{m-2}&s_0^{m-1}
\\
0& 1& 2s_1 &3s_1^2& \cdots&(m-2)s_1^{m-3}&(m-1)s_1^{m-2}
\\
0& 0& 2 &6s_2& \cdots&(m-2)(m-3)s_3^{m-4}&(m-1)(m-2)s_2^{m-3}
\\
\vdots& \vdots&\vdots& \vdots &\ddots&\vdots&\vdots
\\
0& 0& 0&0& \cdots&(m-2)!& (m-1)!s_{m-1}
\\
0& 0& 0&0& \cdots&0& (m-1)!
\\
\end{matrix}
\right)
=
\prod_{j=0}^{m-1} j!
\end{align*}
does not vanish.
\\
(3)
Take $m=3$, $r_0=r_1=0$, $r_2=1$. Then 
$$
\rmD(z_0,z_1,z_2)=
\left|
\begin{matrix}
1&z_0&z_0^2
\\1&z_1&z_1^2
\\0&1&2z_2
\\
\end{matrix}
\right|
=
(z_1-z_0)(2z_2-z_1-z_0).
$$
A polynomial of degree $2$ vanishing at $s$ and $-s$ with $s\not=0$ has a zero derivative at the origin. 
For the study of entire functions $f$ satisfying 
$$
f^{(3n)}(s_0)\in\Z, \; f^{(3n)}(s_1)\in\Z, \; f^{(3n+1)}(s_2)\in\Z \quad\hbox{for}\quad n\ge 0,
$$ 
the assumption $\rmD(\bfs)\not=0$ amounts to $2s_2\not=s_1+s_0$.

\subsection{Interpolation polynomials}
The following interpolation polynomials generalize the sequences of polynomials introduced by Lidstone, Whittaker, Poritsky, Gontcharoff and others. 

\begin{proposition}\label{Prop:InterpolationPolynomials}
Assume $\rmD(\bfs)\not=0$. Then there exists a unique family of polynomials $(\Lambda_{nj}(z))_{n\ge 0,0\le j\le m-1}$ satisfying 
\begin{equation}\label{Equation:Lambda}
\Lambda_{nj}^{(mk+r_\ell)}(s_\ell)=\delta_{j\ell}\delta_{nk}, \quad\hbox{for}\quad n, k\ge 0 \quad\hbox{and}\quad 0\le j, \ell\le m-1.
\end{equation}
For $n\ge 0 $ and $0\le j\le m-1$ the polynomial $\Lambda_{nj}$ has degree $\le mn+m-1$. 
\end{proposition}

This result plays a main role in our paper; we give two proofs of it. 

\begin{proof}[First proof of  Proposition \ref{Prop:InterpolationPolynomials}].
Assuming $\rmD(\bfs)\not=0$, we prove by induction on $n$ that the linear map 
$$
\begin{matrix}
\psi_n:&\C[z]_{\le m(n+1)-1}&\longrightarrow& \C^{m(n+1)} \qquad
\\
&L(z)&\longmapsto& \bigl(L^{(mk+r_\ell)}(s_\ell)\bigr)_{0\le \ell\le m-1, 0\le k\le n}
\end{matrix}
$$
is an isomorphism of $\C$--vector spaces. For $n=0$ this is the assumption \eqref{Equation:D(s)not0}. Assume $\psi_{n-1}$ is injective for some $n\ge 1$. Let $L\in\ker\psi_n$. Then $L^{(m)}\in\ker\psi_{n-1}$, hence $L^{(m)}=0$, which means that $L$ has degree $<m$. From   \eqref{Equation:D(s)not0} we conclude $L=0$. 

The fact that $\psi_n$ is injective for all $n$ implies that if a polynomial $f\in\C[z]$ satisfies $f^{(mk+r_\ell)}(s_\ell)=0$ for all $k\ge 0$ and all $\ell$ with $0\le \ell\le m-1$, then $f=0$. This shows the unicity of the solution $\Lambda_{nj}$ of \eqref{Equation:Lambda}.

Since $\psi_n$ is injective, it is an isomorphism, and hence surjective: for $0\le j\le n-1$ there exists a unique polynomial $\Lambda_{nj}\in \C[z]_{\le m(n+1)-1}$ such that $\Lambda_{nj}^{(mk+r_\ell)}(s_\ell)=\delta_{j\ell}\delta_{nk}$ for $0\le j, \ell\le m-1$. These conditions show that the set of polynomials $\Lambda_{kj}$ for $0\le k\le n$, $0\le j\le m-1$, is a basis of $\C[z]_{\le m(n+1)-1}$: any polynomial $f\in\C[z]$ of degree $\le m(n+1)-1$ can be written in a unique way 
$$
f(z)=\sum_{j=0}^{m-1} \sum_{k= 0} ^n a_{kj}\Lambda_{kj}(z),
$$
and therefore the coefficients are given by $a_{kj}=f^{(mk+r_j)}(s_j)$. 
\end{proof}

\begin{proof}[Second proof of  Proposition \ref{Prop:InterpolationPolynomials}].
The conditions \eqref{Equation:Lambda} mean that any polynomial $f\in\C[z]$ has an expansion 
\begin{equation}\label{Equation:Expansion}
f(z)=\sum_{j=0}^{m-1} \sum_{n\ge 0} f^{(mn+r_j)}(s_j)\Lambda_{nj}(z),
\end{equation}
where only finitely many terms on the right hand side are nonzero. 

Assuming $\rmD(\bfs)\not=0$, we first prove the unicity of such an expansion by induction on the degree of $f$. The assumption $\rmD(\bfs)\not=0$ shows that there is no nonzero polynomial of degree $<m$ satisfying $f^{(mn+r_j)}(s_j)=0$ for all $(n,j)$ with $0\le n,j\le m-1$. Now if $f$ is a polynomial satisfying $f^{(mn+r_j)}(s_j)=0$ for all $(n,j)$ with $n\ge 0$ and $0\le j\le m-1$, then $f^{(m)}$ satisfies the same conditions and has a degree less than the degree of $f$. By the induction hypothesis we deduce $f^{(m)}=0$, which means that $f$ has degree $<m$, hence $f=0$. This proves the unicity. 

For the existence, let us show that, under the assumption $\rmD(\bfs)\not=0$, the recurrence relations 
$$ 
\Lambda_{nj}^{(m)}=\Lambda_{n-1,j}, \quad 
\Lambda_{nj}^{(r_\ell)}(s_\ell)=0 \text{ for $n\ge 1$},
\quad \Lambda_{0j}^{(r_\ell)}(s_\ell)=\delta_{j\ell} \text{ for $0\le j, \ell\le m-1$}
$$ 
have a unique solution given by polynomials $\Lambda_{nj}(z)$, ($n\ge 0$, $j=0,\dots,m-1$), where $\Lambda_{nj}$ has degree $\le mn+m-1$. Clearly, these polynomials will satisfy \eqref{Equation:Lambda}.

From the assumption $\rmD(\bfs)\not=0$ we deduce that, for $0\le j\le m-1$, there is a unique polynomial 
$ \Lambda_{0j}$ of degree $<m$ satisfying
 $$
 \Lambda_{0j}^{(r_\ell)}(s_\ell)=\delta_{j\ell} \text{ for $0\le \ell\le m-1$}.
 $$
By induction, given $n\ge 1$ and $j\in\{0,1,\dots,m-1\}$, once we know $\Lambda_{n-1,j}(z)$, we choose a solution $L$ of the differential equation $L^{(m)}=\Lambda_{n-1,j}$; using again the assumption $\rmD(\bfs)\not=0$, we deduce that there is a unique polynomial $\Ltilde$ of degree $<m$ satisfying $\Ltilde^{(r_\ell)}(s_\ell)=L^{(r_\ell)}(s_\ell)$ for $0\le \ell\le m-1$; then the solution is given by $\Lambda_{nj}=L-\Ltilde$.
\end{proof}

\begin{remark}{\rm
The following converse of Proposition \ref{Prop:InterpolationPolynomials} is plain: if there exists a unique tuple 
$$
\bigl(\Lambda_{00}(z),\Lambda_{01}(z),\dots,\Lambda_{0,m-1}(z)\bigr)
$$ 
of polynomials of degree $\le m-1$ satisfying 
$$
 \Lambda_{0j}^{(r_\ell)}(s_\ell)=\delta_{j\ell} \text{ for $0\le j, \ell\le m-1$},
 $$
then $\rmD(\bfs)\not=0$. 
}
\end{remark}

\medskip
\noindent
{\bf Examples}
\\
Special cases of Proposition \ref{Prop:InterpolationPolynomials} have already been introduced in the literature. 
\\
(1) Lidstone polynomials with $\{0,1\}$ \cite[\S 3.1]{MW-Lidstone2points}:
$$
\hbox{
 $m=2$, $s_0=0$, $s_1=1$, $r_0=r_1=0$, $\Lambda_{n0}(z)=\Lambda_n(1-z)$, $\Lambda_{n1}(z)=\Lambda_n(z)$. }
 $$ 
(2) Lidstone polynomials with $\{s_0,s_1\}$ and $s_0\not=s_1$; with the notation of \cite[\S 3.2]{MW-Lidstone2points}: 
$$
\hbox{
$m=2$, $r_0=r_1=0$, $\Lambda_{n0}(z)=-\Lambdatilde_n(z-s_1)$, $\Lambda_{n1}(z)=\Lambdatilde_n(z-s_0)$.}
$$ 
(3) Whittaker polynomials with $\{0,1\}$; with the notation of \cite[\S 5.1]{MW-Lidstone2points}: 
$$
\hbox{
$m=2$, $s_0=1$, $s_1=0$, $r_0=0$, $r_1=1$, $\Lambda_{n0}(z)=M_n(z)$, $\Lambda_{n1}(z)=M'_{n+1}(z-1)$.
}
$$ 
(4) Whittaker polynomials with $\{s_0,s_1\}$; with the notation of \cite[\S 5.2]{MW-Lidstone2points} (beware that  this reference deals with the even derivatives at $s_0$ and the odd derivatives at $s_1$, while here we impose $r_0\le r_1$):  
$$
\hbox{
$m=2$, $r_0=0$, $r_1=1$, $\Lambda_{n0}(z)=\Mtilde_n(z-s_1)$, $\Lambda_{n1}(z)=\Mtilde'_{n+1}(z-s_0)$.
}
$$
(5) \cite{MR1501639}, 
quoted by \cite[\S 3]{MR0059366}, 
\cite{MR0072947} 
(see also \cite[Chap.~3, \S 4.3]{zbMATH03416363})
: assuming $s_0,s_1,\dots,s_{m-1}$ are pairwise distinct, 
$$
r_0=r_1=\cdots=r_{m-1}=0.
$$ 
(6) \cite{zbMATH02563461}, 
quoted by \cite[\S 4]{MR0059366}, 
 \cite{MR0072947} 
(see also 
\cite[Chap.~3, \S 4.2]{zbMATH03416363}): 
$$
r_j=j \quad\hbox{for}\quad j=0,1,\dots,m-1.
$$

\subsection{Exponential sums, following D.~Roy}\label{SS:ExponentialSums}

This section is due to D.~Roy (private communication). 

Given complex numbers $a_0,a_1,\cdots$ and non negative real numbers $c_0,c_1,\dots$, we write
$$
\sum_{i\ge 0} a_iz^i\preceq_z \sum_{i\ge0} c_iz^i
$$
if $|a_i|\le c_i$ for all $i\ge 0$. In the same way, given two power series $\sum_{i\ge 0,j\ge 0} a_{ij}t^i z^j$ and  $\sum_{i\ge 0,j\ge 0} c_{ij}t^i z^j$ with $a_{ij}\in\C$ and  $c_{ij}\in\R_{\ge 0}$, we write
$$
\sum_{i\ge 0}\sum_{j\ge 0} a_{ij}t^i z^j\preceq_{t,z} \sum_{i\ge 0}\sum_{j\ge 0} c_it^iz^j
$$
if $|a_{ij}|\le c_{ij}$ for all $i,j$.

We first give a quantitative version of Proposition \ref{Prop:InterpolationPolynomials}.

\begin{lemma}\label{Lemma:MajorationLambda-n-j} 
There exists a constant $\Theta>0$ such that 
$$
\Lambda_{nj}(z)\preceq_z \sum_{i=0}^{m(n+1)-1}\frac{\Theta^{m(n+1)-i}}{i!} z^i
$$
for all $n\ge 0$ and $j=0,1,\dots,m-1$. 
\end{lemma}

\begin{proof}
We proceed by induction.
For $n=0$ it suffices to choose $\Theta>0$ sufficiently large so that 
$$
\Lambda_{0j}(z)\preceq_z \sum_{i=0}^{m-1}\frac{\Theta^{m-i}}{i!} z^i
$$
for $j=0,1,\dots,m-1$.  Assume
$$
\Lambda_{n-1,j}(z)\preceq_z \sum_{i=0}^{mn-1}\frac{\Theta^{mn-i}}{i!} z^i
$$
for some integer $n\ge 1$ and for $j=0,1,\dots,m-1$. Fix an integer $j$ and let $L(z)\in\C[z]$ be the polynomial satisfying 
$$
L^{(m)}(z)=\Lambda_{n-1,j}(z)\quad
\hbox{and}
\quad
L(0)=L'(0)=\cdots=L^{(m-1)}(0)=0.
$$
We have 
$$
L(z)\preceq_z \sum_{i=0}^{mn-1}\frac{\Theta^{mn-i}}{(i+m)!} z^{i+m}
=
\sum_{i=m}^{m(n+1)-1}\frac{\Theta^{m(n+1)-i}}{i!} z^i.
$$
Set $A=\max\{1,|s_0|,\dots,|s_{m-1}|\}$. For $\ell=0,1,\dots,m-1$, we have 
$$
|L^{(r_\ell)}(s_\ell)|\le \sum_{i=0}^{mn-1}\frac{\Theta^{mn-i} A^{i+m-r_\ell} }{(i+m-r_\ell)!} 
=
\Theta^{mn} A^{m-r_\ell}\sum_{i=0}^{mn-1}\frac{(A/\Theta)^i}{(i+m-r_\ell)!} 
\le \Theta^{mn} A^m \exp(A/\Theta).
$$
From the isomorphism \eqref{Equation:D(s)not0} it follows that there is a constant $B>0$ such that, for any polynomial $\Ltilde(z)\in\C[z]_{\le m-1}$,
$$
\Ltilde (z)\preceq_z \left(\max_{0\le \ell\le m-1} |\Ltilde^{(r_\ell)}(s_\ell) |\right) B \sum_{i=0}^{m-1} \frac {z^i}{i!}\cdotp
$$
Choosing $\Ltilde(z)$ such that 
$$
\Ltilde^{(r_\ell)}(s_\ell)= L^{(r_\ell)}(s_\ell)
$$
for $\ell=0,1,\dots,m-1$ 
and assuming $\Theta\ge 1$ sufficiently large so that 
$$
\Theta\ge B A^m \exp(A/\Theta),
$$
we get
$$
\Ltilde(z) \preceq_z \Theta^{mn+1} \sum_{i=0}^{m-1} \frac {z^i}{i!} \preceq_z \sum_{i=0}^{m-1} \frac {\Theta^{m(n+1)-i} }{i!} z^i ,
$$
hence 
$$
\Lambda_{nj} (z)= L(z)-\Ltilde(z)\preceq_z \sum_{i=0}^{m(n+1)-1}\frac{\Theta^{m(n+1)-i}}{i!} z^i.
$$
\end{proof}

For $j=0,1,\dots,m-1$ and $z\in\C$, consider the power series $\varphi_j(t,z)\in\C[[t]]$ defined by 
\begin{equation}\label{equation:varphij}
\varphi_j(t,z)=\sum_{n\ge 0} t^{mn+r_j}\Lambda_{nj}(z). 
\end{equation}
From Lemma \ref{Lemma:MajorationLambda-n-j} it follows that we have
$$
\varphi_j(t,z)\preceq_{t,z} \sum_{n\ge 0} \sum_{i=0}^{m(n+1)-1}\frac{\Theta^{m(n+1)-i}}{i!} t^{mn+r_j}z^i,
$$
and therefore the function of two complex variables $(t,z)\mapsto \varphi_j(t,z)$ is analytic in the domain $|t|<  1/ \Theta$, $z\in\C$. 
 
\begin{lemma}\label{Lemma:etz} 
For $|t|<  1/ \Theta$ and $z\in\C$, we have
$$
\rme^{tz}=\sum_{j=0}^{m-1}\rme^{ts_j}\varphi_j(t,z).
$$
\end{lemma}

\begin{proof}
Define, for $|t|<  1/ \Theta$ and $z\in\C$, 
$$
F(t,z)= \sum_{j=0}^{m-1} \rme^{ts_j}\varphi_j(t,z)-\rme^{tz}.
$$
We have 
$$
F(t,z)=\sum_{j=0}^{m-1} \rme^{ts_j} \sum_{n\ge 0} t^{mn+r_j}\Lambda_{nj}(z)-\rme^{tz}
=\sum_{n\ge 0} a_n(z) t^n,
$$
where $a_n(z)\in\C[z]_{\le n+m-1}$ for all $n\ge 0$. We obtain, for all $k\ge 0$ and $\ell=0,1,\dots,m-1$, 
$$
\left.
\left(
\frac \partial {\partial z}
\right)^{mk+r_\ell}
F(t,z)
\right|_{z=s_\ell} 
=
\sum_{j=0}^{m-1} \rme^{ts_j}\sum_{n\ge 0} t^{mn+r_j} \Lambda_{nj}^{(mk+r_\ell)}(s_\ell)
-t^{mk+r_\ell}\rme^{t s_\ell}
=0,
$$
hence
$$
\sum_{n\ge 0} a_n^{(mk+r_\ell)}(s_\ell) t^n=0
$$
for $|t|<  1 /\Theta$. Therefore $a_n^{(mk+r_\ell)}(s_\ell)=0$ for all $k\ge 0$,  $n\ge 0$ and $\ell=0,1,\dots,m-1$. We conclude
$a_n(z)=0$ for all $n\ge 0$, which proves Lemma \ref{Lemma:etz}. 
\end{proof}

For $0<|t|<  1/ \Theta$ and $j=0,1,\dots,m-1$, we have
$$
\left(
\frac \partial {\partial z}
\right)^m\varphi_j(t,z)=\sum_{n\ge 0} t^{mn+r_j}\Lambda_{nj}^{(m)}(z)=\sum_{n\ge 1} t^{mn+r_j}\Lambda_{n-1,j}(z)=
t^m\varphi_j(t,z).
$$
The functions $\varphi_0(t,z),\varphi_1(t,z),\dots,\varphi_{m-1}(t,z)$ are the solutions of the differential equation
$$
f^{(m)}(z)=t^m f(z)
$$
with the initial conditions
\begin{equation}\label{Equation:InitialConditions}
\left(\frac{\partial}{\partial z}\right)^{r_\ell}\varphi_j(t,s_\ell)=t^{r_\ell}\delta_{j\ell} \quad\hbox{for}\quad 0\le j,\ell\le m-1.
\end{equation}
Recall that $\zeta$ is a primitive $m$-th root of unity. The general solution of this differential equation is a linear combination of the functions $\rme^{\zeta^k tz}$ ($k=0,1,\dots,m-1$) with coefficients depending on $t$. 
Hence for $0<|t|<  1/ \Theta$ there exist complex numbers $c_{jk}(t)$ ($j,k = 0,1 \dots,m-1$) such that 
\begin{equation}\label{Equation:definitionvarphi}
\varphi_j(t,z)=\sum_{k=0}^{m-1} c_{jk}(t) \rme^{\zeta^k tz}.
\end{equation}
For $\ell=0,1,\dots,m-1$, this yields 
$$
\sum_{k=0}^{m-1} c_{jk}(t)(\zeta^kt)^\ell=
\left.
\left(
\frac{\partial}{\partial z}\right)^\ell \varphi_j(t,z)\right|_{z=0}\
=
\sum_{n\ge 0} t^{mn+r_j}\Lambda_{nj}^{(\ell)}(0),
$$
and thus we deduce that 
$$
t^\ell \sum_{k=0}^{m-1} c_{jk}(t)\zeta^{k\ell}
\preceq_t \sum_{n\ge 0} \Theta^{m(n+1)-\ell} t^{mn+r_j}.
$$
Since the matrix $(\zeta^{k\ell})_{0\le k,\ell\le m-1}$ is invertible, this shows that the functions $c_{jk}(t)$ are meromorphic for $|t|<  1/ \Theta$ with at most a pole at $t=0$ of order $\le m-1$.
 
\subsection{Analytic continuation of $\varphi_j(t,z)$}
 
From 
 \eqref{Equation:InitialConditions}
 and
 \eqref{Equation:definitionvarphi} we deduce that for $j=0,1,\dots,m-1$ and $0<|t|< 1/ \Theta$, we have
$$
\sum_{k=0}^{m-1} c_{jk}(t) \zeta^{kr_\ell} \rme^{\zeta^k t s_\ell}= \delta_{j\ell} \quad (0\le \ell\le m-1).
$$ 
Hence for $|t|<  1/ \Theta$ the matrix $\bigl(c_{jk}(t)\bigr)_{0\le j, k \le m-1}$ is the inverse of the matrix $M(t)$. 
Recall  (Section \ref{SS:periodicm}) that $\Delta(t)$ is the determinant of the matrix $M(t)=\left(\zeta^{kr_\ell} \rme^{\zeta^k t s_\ell}\right)_{0\le k, \ell\le m-1}$.
 We deduce:
\begin{lemma}\label{Lemma:Deltanot0}
The determinant $\Delta(t)$ 
does not vanish for $0<|t|< 1/ \Theta\cdotp$
\end{lemma}
  
The determinant $\Delta(t)$ defines a nonzero entire function in $\C$.  
We extend the definition of $c_{jk}(t)$ to meromorphic functions in $\C$ by the condition that the matrix $\bigl(c_{jk}(t)\bigr)_{0\le j, k \le m-1}$ is the inverse of the matrix $M(t)$. 
From the assumption in 
Theorem \ref{Theorem:mpointsperiodic} that $\Delta(t)$ does not vanish for $0<|t|<\tau$,  we infer that $c_{jk}(t)$ is analytic in the domain $0<|t|<\tau$.
By means of \eqref{Equation:definitionvarphi}, this defines $\varphi_j(t,z)$ for all $z\in\C$ and for all $t$ with $\Delta(t)\not=0$. In particular the function of two variables  $t\mapsto \varphi_j(t,z)$ is analytic in the domain $|t|<\tau$,  $z\in\C$, and  \eqref{equation:varphij} is valid in this domain.

 \begin{lemma}\label{Lemme:majorationLambda}
Let $\varrho$ satisfy $0<\varrho<\tau$. For $z\in\C$ and $0\le j\le m-1$ we have 
$$
|\Lambda_{nj}(z)|\le \varrho^{-mn-r_j} \sup_{|t|=\varrho} |\varphi_j(t,z)|.
$$
\end{lemma}

\begin{proof}
The Taylor expansion at the origin of the meromorphic function $t\mapsto \varphi_j(t,z)$ is given by the formula \eqref{equation:varphij}, which is therefore valid for $|t|<\tau$. Hence 
$$
\Lambda_{nj}(z)=
\frac 1 {2i\pi} \int_{|t|=\varrho} 
\varphi_j(t,z) t^{-mn-r_j-1}\rmd t.
$$
Lemma \ref{Lemme:majorationLambda} follows.
\end{proof}

\noindent
{\bf Examples}
\\
(1) Lidstone 
\cite[\S 3.1]{MW-Lidstone2points}:
$m=2$, $s_0=0$, $s_1=1$, $r_0=r_1=0$, 
$$
 \varphi_0(t,z)=\frac{\sinh((1-z)t)}{\sinh (t)}, \quad \varphi_1(t,z)=\frac{\sinh(tz)}{\sinh (t)}\cdotp
$$ 
(2) Whittaker 
 \cite[\S 5.1]{MW-Lidstone2points}:
 $m=2$, $s_0=1$, $s_1=0$, $r_0=0$, $r_1=1$, 
$$
 \varphi_0(t,z)=\frac{\cosh(tz)}{\cosh (t)}, \quad \varphi_1(t,z)=\frac{\sinh((z-1)t)}{\cosh (t)}\cdotp
$$ 
(3) Poritsky interpolation -- see 
\cite[\S 3]{MR0059366}: 
$r_0=r_1=\cdots=r_{m-1}=0$.
The condition $\rmD(\bfs)=0$ means that  $s_0,s_1,\dots,s_{m-1}$ are   pairwise distinct. The coefficient of $t^{m(m-1)/2}$ in the Taylor expansion at the origin of $\Delta(t)$ is given by the following formula involving two Vandermonde determinants:
$$
\frac{1}{1!2!\cdots (m-1)!}
\det 
\begin{pmatrix}
1&1&\cdots&1
\\
1&\zeta &\cdots&\zeta^{m-1}
\\
1 &\zeta^2 &\cdots&\zeta^{2(m-1)} 
\\
\vdots&\vdots&\ddots&\vdots
\\
1 &\zeta^{m-1} &\cdots&\zeta^{(m-1)^2} 
\\ 
\end{pmatrix}
\det 
\begin{pmatrix}
1&1&\cdots&1
\\
s_0 &s_1 &\cdots&s_{m-1} 
\\
s_0^2 &s_1^2 &\cdots&s_{m-1}^2
\\
\vdots&\vdots&\ddots&\vdots
\\
s_0^{m-1} &s_1^{m-1} &\cdots &s_{m-1}^{m-1}
\\ 
\end{pmatrix}.
 $$
 Hence $\Delta(t)$ has a zero at the origin of multiplicity $ m(m-1)/2$.
 
 For $0\le j\le m-1$, the order of the zero at $t=0$ of $\Delta(t) \varphi_j(t,z)$ is at least $m(m-1)/2$. 

\medskip\noindent
(4) Gontcharoff interpolation -- see 
\cite[\S 4]{MR0059366}: 
$r_j=j $ for $ j=0,1,\dots,m-1$.
In this case $\Delta(0)$ is the Vandermonde determinant
$$ 
\det 
\begin{pmatrix}
1&1&\cdots&1
\\
1&\zeta &\cdots&\zeta^{m-1}
\\
1 &\zeta^2 &\cdots&\zeta^{2(m-1)} 
\\
\vdots&\vdots&\ddots&\vdots
\\
1 &\zeta^{m-1} &\cdots&\zeta^{(m-1)^2} 
\\ 
\end{pmatrix},
 $$
and hence is not zero. 

\subsection{Laplace transform}\label{SS:Laplace}

The main tool for the proof of Theorem \ref{Theorem:mpointsperiodic} is the following result. 

\begin{proposition}\label{Proposition:expansion}
Assume $\rmD(\bfs)\not=0$ and $\Delta(t)\not=0$ for $0<|t|<\tau$.
Then any entire function $f$ of exponential type $<\tau$ has an expansion of the form \eqref{Equation:Expansion}, where the series in the right hand side 
 is absolutely and uniformly convergent for $z$ on any compact  space in $\C$.
\end{proposition}
 As a consequence: 
 
\begin{corollary} \label{Corollary:PeriodicZeroes} 
Under the assumptions of Proposition \ref{Proposition:expansion}, if an entire function $f$  has exponential type $<\tau$ and satisfies 
$$
f^{(mn+r_j)}(s_j)=0 \text{ for } j=0,\dots,m-1 \text{ and all sufficiently large } n,
$$
then $f$ is a polynomial. 
\end{corollary}

The bound for the exponential type is sharp:  if $\alpha\not=0$ is a zero of $\Delta$, then there exists a transcendental entire function of exponential type $|\alpha|$ satisfying the vanishing conditions of Corollary \ref{Corollary:PeriodicZeroes}; for the proof, see Proposition 9(a) of  \cite{MW-Lidstone2points}.

The strategy for the proof of Proposition \ref{Proposition:expansion} will be to check that for $|t|<\tau$, the function $f_t(z)=\rme^{tz}$ admits the expansion \eqref{Equation:Expansion}, and then to deduce the general case by means of the Laplace transform, which is called the method of the kernel expansion in
\cite{MR0072947}, 
\cite[Chap.~I \S 3]{MR0162914}, 
 \cite[\S 1]{MR0059366}. 

We have $f_t^{(m)}=t^m f_t$ and 
$$
f_t^{(r_j)}(s_j)=t^{r_j}\rme^{ts_j}.
$$

 \begin{proof}[Proof of Proposition \ref{Proposition:expansion}]
Let 
$$
f(z)=\sum_{n\ge 0} \frac{a_n}{n!} z^n
$$ 
be an entire function of exponential type $\tau(f)$. Using \eqref{eq:type}, we deduce that the Laplace transform of $f$, 
$$
F(t)=\sum_{n\ge 0} a_nt^{-n-1}, 
$$
is analytic in the domain $|t|>\tau(f)$. From Cauchy's residue Theorem, it follows that for $\varrho>\tau(f)$ we have
$$
f(z) =\frac{1}{2\pi i} \int_{|t|=\varrho} \rme^{t z}F(t) \rmd t.
$$ 
Hence 
$$
f^{(mn+r_j)}(z) =\frac{1}{2\pi i} \int_{|t|=\varrho} t^{mn+r_j}\rme^{t z}F(t) \rmd t.
$$
Assume $\tau(f)<\tau$. Let $\varrho$ satisfy $ \tau(f)<\varrho<\tau$. We deduce from \eqref{equation:varphij} (which is valid for $|t|<\tau$) and Lemma \ref{Lemma:etz} that, for $|t|=\varrho$, we have
$$
\rme^{tz}=\sum_{j=0}^{m-1} \rme^{ts_j}\varphi_j(t,z)=
\sum_{n\ge 0} \sum_{j=0}^{m-1} \rme^{ts_j}t^{mn+r_j}\Lambda_{nj}(z),
$$
which is the expansion \eqref{Equation:Expansion} for the function $f_t(z)=\rme^{tz}$.

We now use Lemma \ref{Lemme:majorationLambda} and permute the integral and the series to deduce
$$
f(z) = 
\sum_{n\ge 0} \sum_{j=0}^{m-1} 
\left( \frac{1}{2\pi i} \int_{|t|=\varrho} t^{mn+r_j}\rme^{t s_j}F(t) \rmd t\right)\Lambda_{nj}(z)
=\sum_{n\ge 0} f^{(mn+r_j)}(s_j)\Lambda_{nj}(z).
$$
Using again Lemma \ref{Lemme:majorationLambda} together with  \eqref{eq:type}, we check that  the last series is absolutely and uniformly convergent for $z$ on any compact space in $\C$. 
\end{proof}

\begin{proof}[Proof of Theorem \ref{Theorem:mpointsperiodic}]
Let $f$ be an entire function 
satisfying the assumptions of Theorem \ref{Theorem:mpointsperiodic}.
From the assumption \eqref{eq:maingrowthconditionmpoints} and Proposition \ref{Proposition:Polya}, we deduce that for $n$ sufficiently large, we have 
$$
f^{(mn+r_j)}(s_j)=0\text{ for } j=0,\dots,m-1.
$$
Since the exponential type of $f$ is $<\tau$, we deduce from Corollary \ref{Corollary:PeriodicZeroes} that $f$ is a polynomial. 
\end{proof}

\section{Sequence of derivatives at several points} 
\label{S:ComplementarySequences}

 Given a sequence $\bfw=(w_n)_{n\ge 0}$ of complex numbers, we investigate the entire functions $f$ such that the numbers $f^{(n)}(w_n)$ are in $\Z$. Under suitable assumptions, we reduce this question to the case where these numbers all vanish. 
 
\subsection{Abel--Gontcharoff interpolation}
 
We start with any sequence $\bfw=(w_n)_{n\ge 0}$ of complex numbers. Following
 \cite{zbMATH02563461} 
 (see also \cite{MR0069937}, 
\cite{MR1912171}), 
we define a sequence of polynomials $(\Omega_{w_0,w_1,\dots,w_{n-1}})_{n\ge 0}$ in $\C[z]$ 
as follows: we set $\Omega_{\emptyset}=1$, $\Omega_{w_0}(z)=z-w_0$, and, for $n\ge 1$, we define $\Omega_{w_0,w_1,w_2,\dots,w_n}(z)$ as the polynomial of degree $n+1$ which is the primitive of $\Omega_{w_1,w_2,\dots,w_n}$ vanishing at $w_0$. 
For $n\ge 0$, we write $\Omega_{n;\bfw}$ for $\Omega_{w_0,w_1,\dots,w_{n-1}}$, a polynomial of degree $n$ which depends only on the first $n$ terms of the sequence $\bfw$. 
The leading term of $\Omega_{n;\bfw}$ is $  (1/{n!}) z^n$. 
An equivalent definition is
$$
\Omega_{n;\bfw}^{(k)}(w_k)=\delta_{kn}
$$
for $n\ge 0$ and $k\ge 0$. As a consequence, any polynomial $P$ can be written as a finite sum
$$
P(z)=\sum_{n\ge 0} P^{(n)}(w_n)\Omega_{n;\bfw}(z).
$$ 
In particular, for $N\ge 0$ we have 
$$
\frac{z^N}{N!}=
\sum_{n= 0}^N \frac{1}{(N-n)!} w_n^{N-n}
\Omega_{n;\bfw}(z).
$$
This gives an inductive formula defining 
$\Omega_{N;\bfw}$: for $N\ge 0$,
\begin{equation}\label{Equation:inductionformulainterpolation}
\Omega_{N;\bfw}(z)=\frac{z^N}{N!}-
\sum_{n= 0}^{N-1} \frac{1}{(N-n)!} w_n^{N-n}
\Omega_{n;\bfw}(z).
\end{equation}
We also have 
$$
\Omega_{w_0,w_1,\dots,w_n}(z)=\Omega_{0,w_1-w_0,w_2-w_0,\dots,w_n-w_0}(z-w_0).
$$
With $w_0=0$, the first polynomials are given by
\begin{align}
\notag
2!\Omega_{0,w_1}(z)&= (z-w_1)^2-w_1^2, \\
\notag
3!\Omega_{0,w_1,w_2}(z)&=(z-w_2)^3-3(w_1-w_2)^2z+w_2^3,
\\
\notag
4!\Omega_{0,w_1,w_2,w_3}(z)&=(z-w_3)^4-6(w_2-w_3)^2(z-w_1)^2
\\
\notag
& 
\quad -4(w_1-w_3)^3z+6w_1^2(w_2-w_3)^2- w_3^4. 
\end{align} 
From the definition we deduce the following formula, involving iterated integrals
$$
\Omega_{w_0,w_1,\dots,w_{n-1}}(z)=\int_{w_0}^z\rmd t_1 \int_{w_1}^{t_1}\rmd t_2 \cdots \int_{w_{n-1}}^{t_{n-1}}\rmd t_n. 
$$
These polynomials are also given by a determinant \cite[p.~7]{zbMATH02563461} 
$$
\Omega_{w_0,w_1,\dots,w_{n-1}}(z)= 
(-1)^n
\left|
\begin{matrix}
1 &\displaystyle \frac{z}{1!} &\displaystyle\frac{z^2}{2!} & \cdots &\displaystyle \frac{z^{n-1}}{(n-1)!} &\displaystyle\frac{z^n}{n!} \\
1 &\displaystyle \frac{w_0}{1!} &\displaystyle\frac{w_0^2}{2!} & \cdots &\displaystyle \frac{w_0^{n-1}}{(n-1)!} &\displaystyle\frac{w_0^n}{n!} \\
0 &\displaystyle1&\displaystyle \frac{w_1}{1!} &\cdots &\displaystyle \frac{w_1^{n-2}}{(n-2)!} &\displaystyle\frac{w_1^{n-1}}{(n-1)!} \\
0 &0&1& \cdots &\displaystyle \frac{w_2^{n-3}}{(n-3)!} &\displaystyle \frac{w_2^{n-2}}{(n-2)!} \\
\vdots&\vdots&\vdots& \ddots &\vdots& \vdots \\
0 &0&0& \cdots & 1&\displaystyle \frac{w_{n-1} }{1!} \\
\end{matrix}
\right|.
$$ 
With the sequence $\bfw=(1,0,1,0,\dots,0,1,\dots)$, we recover the Whittaker polynomials \cite[\S 5]{MW-Lidstone2points} 
$$
\Omega_{2n;\bfw}(z)=M_n(z), \quad \Omega_{2n+1,\bfw}(z)=M'_{n+1}(z-1).
$$
Another example, considered by N.~Abel
(see \cite{zbMATH02705364}, 
 \cite[p.~7]{zbMATH02563461}, 
 \cite[\S 7]{MR0029985}), 
is the arithmetic progression $\bfw=(a+nt)_{n\ge 0}$ with $a$ in $\C$ and $t$ in $\C\setminus\{0\}$, where
$$
\Omega_{n;\bfw}(z)=\frac 1 {n!} (z-a)(z-a-nt)^{n-1}
$$
for $n\ge 1$, which satisfies
$$
\Omega'_{n;\bfw}(z)=\Omega_{n-1;\bfw}(z-t).
$$
Theorem III in \cite[p.~29]{zbMATH02563461} 
gives sufficient conditions on the sequence $(w_n)_{n\ge 0}$ so that an entire function $f$ satisfying some growth condition has an expansion 
$$
f(z)=\sum_{n\ge 0} f^{(n)}(w_n) \Omega_{n;\bfw}(z).
$$
In the case that we are going to consider where the sequence $(|w_n|)_{n\ge 0}$ is bounded, say $|w_n-w_0|\le r$, the condition 
\cite[(31') p.~33]{zbMATH02563461} 
reduces to $\tau< 1/(\rme r)$. See also 
\cite[\S 10]{zbMATH02543645} 
for an improvement in the case $m=2$. 

From now on we assume that the sequence $(|w_n|)_{n\ge 0}$ is bounded. Let $A>\sup_{n\ge 0} |w_n|$. 

\begin{proposition}\label{Proposition:borneOmega} 
Let $\kappa>1/\log 2$. For $n$ sufficiently large, we have, for all $r\ge |A|$,
$$
|\Omega_{n;\bfw}|_r\le (\kappa r)^n.
$$ 
\end{proposition}

\begin{proof}
Let $c_0,c_1,c_2,\dots$ be the sequence of positive numbers defined by induction as follows:
$c_0=1$ and, for $n\ge 1$, 
$$
c_n=\frac{1}{n!}+\frac{c_0}{n!}+\frac{c_1}{(n-1)!}+\cdots+\frac{c_{n-2}}{2!}+c_{n-1}.
$$
From \eqref{Equation:inductionformulainterpolation} we deduce by induction, for $|z|\le r$ and all $n\ge 0$, 
$$
|\Omega_{n;\bfw}(z)|\le c_n r^n.
$$
Let $\kappa_1$ satisfy $ 1 /\log 2<\kappa_1<\kappa$ and let $A>0$ satisfy
$$
A\ge \left( 2-\rme^{\frac 1 {\kappa_1}}\right)^{-1} \max_{n\ge 0} \frac{1}{\kappa_1^n n!} \cdotp 
$$
One checks by induction $c_n \le A \kappa_1^n$ for all $n\ge 0$ thanks to the upper bound
$$
\frac 1 {n!} +A \kappa_1^n\left (\rme^{\frac 1 {\kappa_1}}-1\right)
\le A \kappa_1^n.
$$ 
Therefore,  for sufficiently large $n$, we have $c_n<\kappa^n$.
\end{proof}

In the case $m=2$ and $w_n\in\{0,1\}$ for all $n\ge 0$, 
a sharper estimate has been achieved in \cite[\S 10]{zbMATH02543645}, 
namely 
$$
|\Omega_{n;\bfw}(z)|\le \frac 1 2 \rme^2\left( \frac 1 2 + R\right)^n
$$
for $|z-\frac 12|=R$. The proof relies on explicit formulae for the polynomials $\Omega_{n;\bfw}(z)$.

From Proposition \ref{Proposition:borneOmega} we deduce the following interpolation formula:

\begin{proposition}\label{Proposition:InterpolationSequence}
Let $f$ be an entire function of exponential type $\tau(f)$ satisfying 
$$
\tau(f)<\frac{\log 2}{A}\cdotp
$$
Let $r$ be a real number in the range
$$
A\le r<\frac{\log 2}{\tau(f)}\cdotp
$$
Then 
$$
f(z)=\sum_{n\ge 0} f^{(n)}(w_n)\Omega_{n;\bfw}(z),
$$ 
where the series on the right hand side is absolutely and uniformly convergent in the disk $|z|\le r$.
\end{proposition}

\begin{proof}
Let $\kappa$ and $\tau$ satisfy 
$$
\kappa>\frac 1 {\log 2}, \quad
\tau(f)<\tau<\frac{1}{\kappa r}\cdotp
$$
Write the Taylor expansion of $f$ at the origin:
$$
f(z)=\sum_{N\ge 0} a_N \frac{z^N}{N!}\cdotp
$$
From \eqref{eq:type} we deduce that there exists a constant $c>0$ such that, for all $N\ge 0$, we have $|a_N|\le c \tau^N$. For $|z|\le r$, we have 
$$
|a_N| \sum_{n= 0}^N \left| \frac{1}{(N-n)!} w_n^{N-n}
\Omega_{n;\bfw}(z)\right|\le 
c\tau^N\sum_{n=0}^N \frac{A^{N-n}(\kappa r)^n}{(N-n)!}
\le c \rme^{A/\kappa r} (\tau\kappa r)^N ,
$$
which is the general term of a convergent series, since $\tau\kappa r<1$. Hence 
\begin{align}
\notag
f(z)&=\sum_{N\ge 0} a_N \sum_{n= 0}^N \frac{1}{(N-n)!} w_n^{N-n}
\Omega_{n;\bfw}(z)
\\
\notag
&= \sum_{n\ge 0} 
\Omega_{n;\bfw}(z)\sum_{N\ge n} a_N \frac{1}{(N-n)!} w_n^{N-n}
= \sum_{n\ge 0} 
\Omega_{n;\bfw}(z)f^{(n)}(w_n).
\end{align}
\end{proof}

\begin{remark}{\rm
Notice that here the expansions are valid in a bounded domain of $\C$, not in the entire complex plane as in Section \ref{SS:Laplace} for instance.
}
\end{remark}

\begin{corollary}\label{Corollary:fnwn=0}
If an entire function $f$ of exponential type $\tau(f)< \log 2/A$ satisfies 
$f^{(n)}(w_n)=0$ for all sufficiently large $n$, then $f$ is a polynomial. 
\end{corollary}

Replacing $z$ by $Az$, one may assume $A=1$, and then Corollary \ref{Corollary:fnwn=0} is 
\cite[Theorem 8]{zbMATH02532117}, 
a special case of one of Takenaka's theorems. 

In the special case where the set $\{w_0,w_1,w_2,\dots\}$ is finite, say $S=\{s_0,s_1,\dots,s_{m-1}\}$ with 
$$
\max \{|s_0|,|s_1|,\dots,|s_{m-1}|\}<A, 
$$
Corollary \ref{Corollary:fnwn=0} reduces to the following statement: 

\begin{corollary}\label{Corollary:produitfnsj=0}
If an entire function $f$ of exponential type $\tau(f)< \log 2/A$ satisfies 
$$
\prod_{j=0}^{m-1}f^{(n)}(s_j)=0
$$ 
for all sufficiently large $n$, then $f$ is a polynomial. 
\end{corollary}

\subsection{Sequence of elements in $S$} 

\begin{proof}[Proof of Theorem \ref{Th:mPointsAlln}]
Denote by $\tau(f)$ the exponential type of $f$. 
Since $f$ satisfies the hypothesis \eqref{eq:maingrowthconditionA} of Proposition \ref{Proposition:Polya}, for $n$ sufficiently large we have 
$|f^{(n)}(z)|<1$ for all $|z|<A$. 
 
Under the assumption (a) of Theorem \ref{Th:mPointsAlln}, for $n$ sufficiently large we have $f^{(n)}(w_n)=0$. Corollary \ref{Corollary:fnwn=0} implies that $f$ is a polynomial. 
 
For each sufficiently large  $n$, the product $\prod_{j\in I_n} f^{(n)}(s_j)$ is an integer of absolute value less than $1$, and hence it vanishes. 
Part (b) of 
Theorem \ref{Th:mPointsAlln} follows from Corollary \ref{Corollary:produitfnsj=0}.
\end{proof}

\section*{Acknowledgment}
A preliminary version of this paper was substantially improved thanks to a contribution by Damien Roy (see Section \ref{SS:ExponentialSums}).

 \vskip 1truecm plus .5truecm minus .5truecm 
 
\noindent
Received 30 November 2019.
Revised 1 May 2020.

 \vskip 1truecm plus .5truecm minus .5truecm 
  
 \noindent
\vbox{\hbox{\sc Michel WALDSCHMIDT:} 
\hbox{\url{michel.waldschmidt@imj-prg.fr}}
\hbox{Sorbonne Universit\'e, Faculté Sciences et Ingénierie} 
	\hbox{CNRS, Institut Mathématique de Jussieu Paris Rive Gauche IMJ-PRG, 75005 Paris, France }
	\hbox{Url: \url{http://www.imj-prg.fr/~michel.waldschmidt}}	
}	

\vfill\eject

\begin{thebibliography}{}
\small

\bibitem[Boas and Buck 1964]{MR0162914}
R.~P. Boas, Jr.  and R.~C. Buck, 
 {\em Polynomial expansions of analytic functions},
Ergebnisse der Mathematik und ihrer
 Grenzgebiete, N.F. {\bf 19}, Academic Press,  
 New York, 1964. 
 \href{http://msp.org/idx/mr/0162914}{MR} 
 \href{http://msp.org/idx/zbl/0082.05702}{Zbl} 
 
\bibitem[Buck 1948]{MR0029985}
R.~C. Buck,  
 \href{http://dx.doi.org/10.2307/1990503}{
``Interpolation series''},
 {\em Trans. Amer. Math. Soc.} {\bf 64}  (1948), 283--298.
 \href{http://msp.org/idx/mr/29985}{MR} 
 \href{http://msp.org/idx/zbl/0033.36401}{Zbl} 

\bibitem[Buck 1955]{MR0072947}
R.~C. Buck,  
 \href{http://dx.doi.org/10.2307/2032936}{
``On {$n$}-point expansions of entire functions''},
 {\em Proc. Amer. Math. Soc.} {\bf 6} (1955), 793--796.
 \href{http://msp.org/idx/mr/72947}{MR} 
 \href{http://msp.org/idx/zbl/0068.05502}{Zbl}


\bibitem[Evgrafov 1954]{MR0069937}
M.~A. Evgrafov, 
 {\em Interpolyacionnaya zada\v{c}a {A}belya-{G}on\v{c}arova},
 Gosudarstv. Izdat. Tehn.-Teor. Lit., Moscow, 1954.
 \href{http://msp.org/idx/mr/0069937}{MR} 
 \href{http://msp.org/idx/zbl/0058.05902}{Zbl}
 


\bibitem[{Gel'fond} 1971]{zbMATH03416363}
A.~O. Gel'fond, 
 {\em Calculus of finite differences}, 
Hindustan Publishing Corp., Delhi,  1971.  
\href{http://msp.org/idx/mr/0342890}{MR} 
\href{http://msp.org/idx/zbl/0264.39001}{Zbl}

\bibitem[{Gontcharoff} 1930]{zbMATH02563461}
W. Gontcharoff,  
 \href{http://dx.doi.org/10.24033/asens.798}{
``Recherches sur les d\'eriv\'ees successives des fonctions
 analytiques. G\'en\'eralisation de la s\'erie d'Abel''},
 {\em Ann. Sci. \'Ecole Norm. Sup.} (3) {\bf 47}  (1930), 1--78.
 \href{http://msp.org/idx/mr/1509300}{MR} 
 \href{http://msp.org/idx/zbl/56.0260.01}{Zbl}


\bibitem[{Halph\'en} 1882]{zbMATH02705364}
G. Halph\'en,  
 \href{http://www.numdam.org/item?id=BSMF_1882__10__67_1}{
``Sur une s\'erie d'Abel''},
 {\em Bull. Soc. Math. France} {\bf 10} (1882), 67--87.
 \href{http://msp.org/idx/mr/1503886}{MR} 
 \href{http://msp.org/idx/jfm/14.0367.02}{JFM} 
 
 

\bibitem[Macintyre 1954]{MR0059366}
A.~J. Macintyre,  
 \href{http://dx.doi.org/10.2307/1990741}{
``Interpolation series for integral functions of exponential type''}, 
 {\em Trans. Amer. Math. Soc.} {\bf 76} (1954), 1--13.
 \href{http://msp.org/idx/mr/59366}{MR} 
 \href{http://msp.org/idx/zbl/0058.29702}{Zbl}



\bibitem[Popov 2002]{MR1912171}
A.~Y. Popov,  
 \href{http://dx.doi.org/10.4213/sm629}{
``Bounds for the convergence and uniqueness of {A}bel-{G}oncharov
 interpolation problems''},
 {\em Mat. Sb.} {\bf 193}:2, 97--128.
In Russian; translated in 
 \href{https://doi.org/10.1070/SM2002v193n02ABEH000629}
{Sb. Math. {\bf 193}:2 (2002), 247–277}.
 \href{http://msp.org/idx/mr/1912171}{MR} 
 \href{http://msp.org/idx/zbl/1046.30018}{Zbl}

\bibitem[Poritsky 1932]{MR1501639}
H. Poritsky,  
 \href{http://dx.doi.org/10.2307/1989543}{
``On certain polynomial and other approximations to analytic functions''},
 {\em Trans. Amer. Math. Soc.} {\bf 34}:2 (1932), 274--331.
 \href{http://msp.org/idx/mr/1501639}{MR} 
 \href{http://msp.org/idx/zbl/0004.34307}{Zbl}

\bibitem[Waldschmidt 2019]{MW-Lidstone2points}
M. Waldschmidt, 
 \href{http://www.imj-prg.fr/~michel.waldschmidt/articles/pdf/IntegerValuedDerivativesTwoPoints.pdf}{
``On transcendental entire functions with infinitely many derivatives
 taking integer values at two points''},
 preprint, 2019.
 \href{http://msp.org/idx/arx/1912.00173}{arXiv: 1912.00173} 

\bibitem[{Whittaker} 1934]{zbMATH02543645}
J.~M. Whittaker,  
 \href{http://dx.doi.org/10.1112/plms/s2-36.1.451}{
``On Lidstone's series and two-point expansions of analytic
 functions''},
 {\em Proc. Lond. Math. Soc. (2)} {\bf 36} (1934), 451--469.
 \href{http://msp.org/idx/mr/1575969}{MR} 
 \href{http://msp.org/idx/zbl/59.0358.02}{Zbl}

\bibitem[{Whittaker} 1964]{zbMATH02532117}
J.~M. Whittaker, 
 {\em Interpolatory function theory},    
Cambridge Tracts in Mathematics and Mathematical Physics {\bf 33},
Stechert-Hafner, New York, 1964.
 \href{http://msp.org/idx/mr/0185330}{MR} 
 \href{http://msp.org/idx/zbl/0012.15503}{Zbl}

\end{thebibliography}
\end{document}